\documentclass{birkjour}
\usepackage{amsmath}
\usepackage{amssymb}
\usepackage{amsthm}

\newtheorem{thm}{Theorem}[section]

\newtheorem{lem}[thm]{Lemma}
\newtheorem{prop}[thm]{Proposition}
\theoremstyle{definition}

\theoremstyle{remark}
\newtheorem{rem}[thm]{Remark}
\newtheorem*{ex}{Example}
\numberwithin{equation}{section}
\usepackage[colorlinks, linkcolor=blue, anchorcolor=blue, citecolor=blue]{hyperref}
\usepackage{color}

\begin{document}
	
	%
	%
	%
	%
	%
	%
	%
	%
	%

	\title[variable exponent Bergman space]
	{Toeplitz operators and weighted composition operators on variable exponent Bergman spaces}

	\author[Cezhong Tong]{Cezhong Tong}
	
	\address{
		Department of Mathematics\\
		Hebei University of Technology\\
		Tianjin 300401 China
	}
	
	\email{ctong@hebut.edu.cn}
	
	\author[Zicong Yang]{Zicong Yang}
	
	\address{
		Department of Mathematics\\
		Hebei University of Technology\\
		Tianjin 300401 China
	}
	
	\email{zicongyang@126.com; zc25@hebut.edu.cn}
	
	\author[Zehua Zhou]{Zehua Zhou}
	
	\address{
		School of Mathematics\\
		Tianjin University\\
		Tianjin 300354 China
	}
	
	\email{zehuazhoumath@aliyun.com}
	
	\thanks{Tong was supported in part by the National Natural Science Foundation of China (Grant No. 12171136, 12411530045). Yang was supported in part by the Natural Science Foundation of Hebei Province (Grant No. A2023202031, A2023202037). Zhou was supported in part by the National Natural Science Foundation of China (Grant No. 12171353).}
	
	
	
	\begin{abstract}
		In a recent paper [JFA, 278 (2020), 108401], Choe et al. obtained characterizations for bounded and compact differences of two weighted composition operators acting on standard weighted Bergman spaces over the unit disk in terms of Carleson measures. Then they extended the results to the ball setting. In this paper, we further generalize those results to variable exponent Bergman spaces over the unit ball. Our proofs, when restricted to the case of constant variable, are new and simpler. Moreover, boundedness and compactness of Toeplitz operators on variable exponent Bergman spaces are also characterized.
	\end{abstract}
	
	\maketitle
	\section{Introduction}
	
	\subsection{Variable exponent Bergman space}
	
	The theory of variable exponent spaces has witnessed an explosive growth in recent years. The study of such spaces has an intrinstic interest and has a wide variety of applications, such as in differential equations and minimization problems with non-standard growth. Variable exponent Lebesgue spaces are generalizations of  classical Lebesgue spaces where the exponent is a measurable function and thus the exponent may vary. It seems to appear in the literature for the first time in a 1931 paper by Orlicz \cite{Ow}. One can refer to the monographs \cite{CdFa,DlHpHpRm} for more history and some real analysis theory about variable exponent Lebesgue spaces. 
	
	The variable exponent Bergman space over the unit disk was first introduced by Chac\'on and Rafeiro \cite{CgRh}. Research on variable exponent Bergman spaces and operators acting on them is in fact at the very beginning, see \cite{CgRh1,CgRhVj, Da, Ft, YzZz}. In this paper, we aim to study Toeplitz operators and weighted composition operators on variable exponent Bergman spaces. We first recall some basic notations. 
	
	Let $\mathbb{C}^n$ be the $n$-dimensional complex Euclidean space and $\mathbb{B}_n$ be the open unit ball in $\mathbb{C}^n$. For any two points $z=(z_1,\cdots,z_n)$ and $w=(w_1,\cdots,w_n)$ in $\mathbb{C}^n$
	, we write $\langle z,w\rangle=\sum_{i=1}^n z_i\cdot\overline{w}_i$ and $|z|=\langle z,z\rangle ^{1/2}$. Given a positive Borel measure $\mu$ on $\mathbb{B}_n$, a measurable function $p:\mathbb{B}_n\to [1,\infty)$ is called a variable exponent. Denote by $p^+:={\rm ess}\sup_{z\in\mathbb{B}_n}p(z)$ and $p^-:={\rm ess}\inf_{z\in\mathbb{B}_n}p(z)$. For a complex-valued function $f$ on $\mathbb{B}_n$, we define the modular $\rho_{p(\cdot),\mu}$ by 
	\[\rho_{p(\cdot),\mu}(f)=\int_{\mathbb{B}_n}|f(z)|^{p(z)}d\mu(z),\]
	and the Luxemburg-Nakano norm by 
	\[\|f\|_{p,\mu}=\inf\left\{\gamma>0:\rho_{p(\cdot),\mu}\left(\frac{f}{\gamma}\right)\leq 1\right\}.\]
	
	There are some close connections between the modular and the norm, see  \cite{DlHpHpRm} for example. 
	\begin{lem}\label{lemma1.1}
		Suppose $p^+<\infty$, then the following conditions hold:
		\begin{itemize}
			\item [(i)] $\|f\|_{p(\cdot),\mu}\leq 1$ and $\rho_{p(\cdot),\mu}\leq 1$ are equivalent, as are $\|f\|_{p(\cdot),\mu}=1$ and $\rho_{p(\cdot),\mu}(f)=1$. 
			\item [(ii)] $\|f\|_{p(\cdot),\mu}\leq \rho_{p(\cdot),\mu}(f)+1$.
			\item [(iii)] The modular convergence and norm convergence are equivalent.
		\end{itemize}
	\end{lem}
	
	For $p^+<\infty$, the variable exponent Lebesgue space $L^{p(\cdot)}(\mathbb{B}_n,d\mu)$ consists of all complex-valued measurable functions $f$ such that $\rho_{p(\cdot),\mu}(f)<\infty$. It is a Banach space equipped with the Luxemburg-Nakano norm. Let $dV$ be the normalized Lebesgue measure on $\mathbb{B}_n$ such that $V(\mathbb{B}_n)=1$. When dealing with the measure $dV$, we will drop the subscript $\mu$ and write the modular simply as $\rho_{p(\cdot)}$ and the norm as $\|\cdot\|_{p(\cdot)}$ respectively. The variable exponent Bergman space
	\[A^{p(\cdot)}(\mathbb{B}_n)=L^{p(\cdot)}(\mathbb{B}_n,dV)\cap H(\mathbb{B}_n),\]
	where $H(\mathbb{B}_n)$ is the space of all holomorphic functions on $\mathbb{B}_n$. It is easy to show that $A^{p(\cdot)}(\mathbb{B}_n)$ is a closed subspace of $L^{p(\cdot)}(\mathbb{B}_n,dV)$. When $p(\cdot)$ is a constant, it reduces the classical Bergman space $A^p(\mathbb{B}_n)$, see \cite{ZrZk} for details.
	
	A function $p:\mathbb{B}_n\to\mathbb{R}$ is said to be log-H${\rm\ddot{o}}$lder continuous or satisfy the Dini-Lipschitz condition if there exists a positive constant $C_{log}$ such that
	\[|p(z)-p(w)|\leq\frac{C_{log}}{\log\frac{1}{|z-w|}}\]
	for all $z,w\in\mathbb{B}_n$ with $|z-w|<\frac{1}{2}$. The set of all log-H${\rm\ddot{o}}$lder continuous functions $p(\cdot)$ on $\mathbb{B}_n$ with $1<p^-\leq p^+<\infty$ will be denoted by $\mathcal{P}^{log}(\mathbb{B}_n)$.
	
	Notice that $A^{p(\cdot)}(\mathbb{B}_n)\subset A^{1}(\mathbb{B}_n)$. According to \cite[Theorem 2.7]{Zk}, one has the following reproducing property
	\begin{equation}\label{equa1.1}
		f(z)=\int_{\mathbb{B}_n}f(w)\overline{K_{z}(w)}dV(w)
	\end{equation}
	for every $f\in A^{p(\cdot)}(\mathbb{B}_n)$ and $z\in\mathbb{B}_n$, where
	\begin{equation}\label{equa1.2}
		K_{z}(w)=\frac{1}{(1-\langle w,z\rangle)^{n+1}}.
	\end{equation}
	The Bergman projection operator $P$ is defined for functions $f$ on $\mathbb{B}_n$ by
	\[Pf(z)=\int_{\mathbb{B}_n}\frac{f(w)}{(1-\langle z,w\rangle)^{n+1}}dV(w),\quad f\in L^{p(\cdot)}(\mathbb{B}_n,dV).\]
	And let 
	\[\hat{P}f(z)=\int_{\mathbb{B}_n}\frac{f(w)}{|1-\langle z,w\rangle|^{n+1}}dV(w),\quad f\in L^{p(\cdot)}(\mathbb{B}_n,dV).\]
	\cite[Theorem 1.6]{BdTeZa} tells us that if $p(\cdot)$ is log-H${\rm\ddot{o}}$lder continuous on $\mathbb{B}_n$, then the Bergman projection $P$ is bounded from $L^{p(\cdot)}(\mathbb{B}_n,dV)$ onto $A^{p(\cdot)}(\mathbb{B}_n)$ and $\hat{P}$ is bounded on $L^{p(\cdot)}(\mathbb{B}_n,dV)$. Consequently, by a simple modification as in the classical case, when $p(\cdot)\in\mathcal{P}^{log}(\mathbb{B}_n)$, the dual space of $A^{p(\cdot)}(\mathbb{B}_n)$ can be identified with $A^{p'(\cdot)}(\mathbb{B}_n)$, where $\frac{1}{p(\cdot)}+\frac{1}{p'(\cdot)}=1$. Every element $\phi\in (A^{p(\cdot)}(\mathbb{B}_n))^{*}$ is associate to a function $g\in A^{p'(\cdot)}(\mathbb{B}_n)$ in such a way
	\[\phi(f)=\int_{\mathbb{B}_n}f(z)\overline{g(z)}dV(z),\quad f\in A^{p(\cdot)}(\mathbb{B}_n),\]
	and $\|\phi\|\simeq \|g\|_{p'(\cdot)}$.

	\subsection{Carleson measure and Toeplitz operator}
	
	Let $\mu$ be a positive Borel measure on $\mathbb{B}_n$. We say that $\mu$ is a Carleson measure for $A^{p(\cdot)}(\mathbb{B}_n)$ if there exists a constant $C>0$ such that
	\[\|f\|_{p(\cdot),\mu}\leq C\|f\|_{p(\cdot)}\]
	for all $f\in A^{p(\cdot)}(\mathbb{B}_n)$. That is, $\mu$ is a Carleson measure for $A^{p(\cdot)}(\mathbb{B}_n)$ if the embedding $A^{p(\cdot)}\subset L^{p(\cdot)}(\mathbb{B}_n,d\mu)$ is continuous. If, in addition, the embedding $A^{p(\cdot)}(\mathbb{B}_n)\subset L^{p(\cdot)}(\mathbb{B}_n,d\mu)$ is compact, then $\mu$ is said to be a compact Carleson measure for $A^{p(\cdot)}(\mathbb{B}_n)$.
	
	According to \cite[Theorem 3.4.7]{DlHpHpRm}, the space $L^{p(\cdot)}(\mathbb{B}_n,d\mu)$ is reflexive when $p(\cdot)\in\mathcal{P}^{\log}(\mathbb{B}_n)$. And the linear combinations of the evaluation functionals are dense in $A^{p'(\cdot)}(\mathbb{B}_n)$. So we conclude that $\mu$ is a compact Carleson measure for $A^{p(\cdot)}(\mathbb{B}_n)$ if and only if 
	\[\|f_j\|_{p(\cdot),\mu}\to 0\]
	for any bounded sequence $\{f_j\}_{j=1}^{\infty}$ in $A^{p(\cdot)}(\mathbb{B}_n)$ that converges to 0 uniformly on compact subsets of $\mathbb{B}_n$.
	
	The concept of Carleson measures was first introduced by L. Carleson \cite{Cl} to prove the corona theorem and to solve interpolation problems for the algebra of all bounded holomorphic functions on the unit disk. Later, Hastings \cite{Hw} characterized the Bergman-Carleson measures. Luecking \cite{Ld} considered the hyperbolic geometry of the disk and characterized Bergman-Carleson measure in terms of the measure on pseudo-hyperbolic disks. By now, Carleson measures are a powerful tool for the study of function spaces and operators acting on them, and have been extended to more general setting. See for example \cite{PjZr, Wz}. Recently, Bergman-Carleson measures with variable exponent were characterized in \cite{BdTeZa, CgRhVj}.
	
	Given $\beta\geq 0$ and a positive Borel measure $\mu$ on $\mathbb{B}_n$, define the Toeplitz operator $T_{\mu}^{\beta}$ as follows:
	\[T_{\mu}^{\beta}f(z)=\int_{\mathbb{B}_n}\frac{f(w)}{(1-\langle z,w\rangle)^{n+1+\beta}}d\mu(w),\quad z\in\mathbb{B}_n.\]
	An account of the theory of Toeplitz operators acting on Bergman spaces can be found in \cite{Zk1}. Pau and Zhao \cite{PjZr} studied the boundedness of the Toeplitz operator $T_{\mu}^{\beta}$ between standard weighted Bergman spaces in terms of Carleson measures. There are also many works that focus on Toeplitz operators acting on various weighted Bergman spaces, see \cite{LxAh, PjRj} and the references therein. Our first aim in this paper is to characterize the boundedness and compactness of $T_{\mu}^{\beta}$ on $A^{p(\cdot)}(\mathbb{B}_n)$, see Theorem 3.2 and Theorem 3.3 below. 
	
	\subsection{Weighted composition operator}
	
	Denote by $S(\mathbb{B}_n)$ the set of all holomorphic self-maps of $\mathbb{B}_n$. Given $u\in H(\mathbb{B}_n)$ and $\varphi\in S(\mathbb{B}_n)$, the weighted composition operator $C_{u,\varphi}$ on $H(\mathbb{B}_n)$ is defined by 
	\[C_{u,\varphi}f=u\cdot f\circ\varphi,\quad f\in H(\mathbb{B}_n).\]
	It is known that weighted composition operators are closely related to the isometries on classical Hardy or Bergman spaces. See for example \cite{Kc}. When $u=1$, it reduces to the composition operator $C_{\varphi}$. The relationship between the operator-theoretic properties of $C_{\varphi}$ and the function-theoretic properties of $\varphi$ has been studied extensively during the past several decades. One can refer to the standard reference \cite{CcMb} for various aspects on the theory of composition operators.
	
	In the study of the isolation phenomena in the space of composition operators acting on Hardy space, Shapiro and Sundberg \cite{SjSc} questioned whether two composition operators are in the same path component when their difference is compact. In 2005, Moorhouse \cite{Mj} characterized the compact difference $C_{\varphi}-C_{\psi}$ on standard weighted Bergman spaces and answered the Shapiro-Sundberg question in the negative. By using Joint-Carleson measures, Koo and Wang \cite{KhWm} studied the bounded and compact differences $C_{\varphi}-C_{\psi}$ in $A_{\alpha}^p(\mathbb{B}_n)$. In 2017, Acharyya and Wu \cite{AsWz} obtained a compactness criteria for $C_{u,\varphi}-C_{v,\psi}$ on weighted Bergman spaces. However, they restricted the weights $u,v$ to satisfy a certain growth condition. Recently, Choe et al. \cite{CbCkKhYj} completely characterized the bounded and compact differences $C_{u,\varphi}-C_{v,\psi}$ on weighted Bergman spaces over the unit disk in terms of Carleson measures under an extra $L^p$-condition for $u$ and $v$. And then they extended the results to the ball setting and removed the $L^p$-condition of $u$ and $v$, see \cite{CbCkKhPi}. To the best of our knowledge, most of the studies of the differences of two weighted composition operators rely on the method in \cite{CbCkKhPi, CbCkKhYj}, which involves certain technical lemmas, see for example \cite{CbCkKhYj1, Lc}. So our second aim in this paper is to construct a new method and generalize those results to the variable exponent setting. Our proofs, when restricted to the constant variable case, are new and simpler, see Theorem 4.5 and Theorem 4.7 below.
	
	This paper is organized as follows. In Section 2, we present some preliminary facts and auxiliary lemmas that will be used later. Section 3 is devoted to describing the boundedness and compactness of the Toeplitz operator $T_{\mu}^{\beta}$. In Section 4, we investigate the properties of $C_{u,\varphi}$ on $A^{p(\cdot)}(\mathbb{B}_n)$. We show that there exist $p_0(\cdot)\in\mathcal{P}^{log}(\mathbb{B}_n)$, $u_0\in H(\mathbb{B}_n)$ and $\varphi_0\in S(\mathbb{B}_n)$ such that $C_{u_0,\varphi_0}$ is bounded on every $A^p(\mathbb{B}_n)$, but not on $A^{p_0(\cdot)}(\mathbb{B}_n)$. And then we characterize the bounded and compact differences $C_{u,\varphi}-C_{v,\psi}$ on $A^{p(\cdot)}(\mathbb{B}_n)$.
	
	Throughout the paper we use the same letter $C$ to denote positive constants which may vary at different occurrences but do not depend on the essential argument. For non-negative quantities $A$ and $B$, we write $A\lesssim B$ (or equivalently $B\gtrsim A$) if there exists an absolute constant $C>0$ such that $A\leq CB$. $A\simeq B$ means both $A\lesssim B$ and $B\lesssim A$.
	
	\section{Preliminaries}
	
	In this section we recall some basic facts and present some auxiliary lemmas which will be used in the sequel.
	
	Given $z\in\mathbb{B}_n$, let $\sigma_z$ be the involutive automorphism of $\mathbb{B}_n$ that exchanges $0$ and $z$. More explicitly,
	\[\sigma_{z}(w)=\frac{z-P_z(w)}{1-\langle w,z\rangle}+\sqrt{1-|z|^2}\frac{P_z(w)-w}{1-\langle w,z\rangle},\quad w\in\mathbb{B}_n,\]
	where $P_z$ is the orthogonal projection from $\mathbb{C}^n$ onto the one dimensional subspace generated by $z$. The pseudo-hyperbolic distance between $z,w\in\mathbb{C}_n$ is given by 
	\[d(z,w)=|\sigma_{z}(w)|.\]
	It is easy to check that
	\begin{equation}\label{equa2.1}
		1-d(z,w)^2=\frac{(1-|z|^2)(1-|w|^2)}{|1-\langle z,w\rangle|^2}.
	\end{equation}
	The pseudo-hyperbolic ball centered at $z\in\mathbb{B}_n$ with radius $s\in (0,1)$ is defined by
	\[E(z,s)=\{w\in\mathbb{B}_n:d(z,w)<s\}.\]
	Given $s\in (0,1)$, it is well-known that 
	\begin{equation}\label{equa2.2}
		V(E(z,s))\simeq (1-|z|^2)^{n+1},
	\end{equation}
	and
	\begin{equation}\label{equa2.3}
		1-|z|^2\simeq 1-|w|^2\simeq |1-\langle z,w\rangle|
	\end{equation}
	for all $z\in\mathbb{B}_n$ and $w\in E(z,s)$. Moreover,
	\begin{equation}\label{equa2.4}
		|1-\langle z,a\rangle|\simeq |1-\langle w,a\rangle|
	\end{equation}
	for all $a,z$ and $w$ in $\mathbb{B}_n$ with $d(z,w)<s$. Here, all the constants suppressed depend only on $n$ and $s$.
	
	The Bergman metric between $z,w\in\mathbb{B}_n$ is given by 
	\[\beta(z,w)=\frac{1}{2}\log\frac{1+d(z,w)}{1-d(z,w)}.\]
	Write $B(z,r)=\{w\in\mathbb{B}_n:\beta(z,w)<r\}$ for the Bergman metric ball centered at $z$ with radius $r>0$. Clearly, $E(z,s)=B(z,r)$ for $s=\tanh(r)$.
	
	\begin{lem}\label{lemma2.1}
		Let $r>0$ and $p(\cdot)\in\mathcal{P}^{log}(\mathbb{B}_n)$. Then
		\[(1-|a|^2)^{p(z)}\simeq (1-|a|^2)^{p(w)}\]
		for all $a\in\mathbb{B}_n$ and $z,w\in B(a,r)$.
	\end{lem}
	
	\begin{proof}
		Assume $r>0$ and $z,w\in B(a,r)$, then $\beta(z,w)<2r$. By \eqref{equa2.3}, we have
		\begin{equation}\label{equa2.5}
			1-|a|^2\simeq 1-|z|^2\simeq |1-\langle z,w\rangle|.
		\end{equation}
		Since
		\[|\sigma_z(w)|^2=\frac{|z-P_z(w)|^2+(1-|z|^2)|w-P_z(w)|^2}{|1-\langle z,w\rangle|^2}\leq 1,\]
		together with \eqref{equa2.5}, we get
		\begin{equation*}
			\begin{split}
				|z-w|&\leq |z-P_z(w)|+|w-P_z(w)|\\
				&\lesssim |1-\langle w,z\rangle|+|1-\langle w,z\rangle|^{1/2}\\
				&\lesssim |1-\langle w,z\rangle|^{1/2}.
			\end{split}
		\end{equation*}
		Then it follows from the log-H${\rm\ddot{o}}$lder continuity of $p(\cdot)$ that 
		\begin{equation*}
			\begin{split}
				|p(z)-p(w)|\log\frac{1}{1-|a|^2}&\lesssim \frac{4C_{log}}{\log\frac{4}{|z-w|}}\log\frac{4}{|1-\langle z,w\rangle|}\\
				&\lesssim \frac{4C_{log}}{\log\frac{4}{|1-\langle z,w\rangle|^{1/2}}}\log\frac{4}{|1-\langle z,w\rangle|}\\
				&\lesssim 1.
			\end{split}
		\end{equation*}
		Hence $(1-|a|^2)^{-|p(z)-p(w)|}\leq e^{|p(z)-p(w)|\log\frac{1}{1-|a|^2}}\lesssim 1$. This shows exactly that $(1-|a|^2)^{p(z)}\simeq (1-|a|^2)^{p(w)}$.
	\end{proof}
	
	The following Jensen type inequality was proved in \cite{RhSs} in the context of spaces of homogeneous type. For the sake of completeness, we present the proof in detail.
	
	\begin{lem}\label{lemma2.2}
		Let $r>0$ and $p(\cdot)\in \mathcal{P}^{log}(\mathbb{B}_n)$. Then
		\begin{equation*}
			\left(\frac{1}{V(B(a,r))}\int_{B(a,r)}|f(w)|dV(w)\right)^{p(z)}\lesssim \frac{1}{V(B(a,r))}\int_{B(a,r)}|f(w)|^{p(w)}dV(w)+1
		\end{equation*} 
		for all $a\in\mathbb{B}_n$ and $z\in B(a,r)$, provided that $\|f\|_{p(\cdot)}\leq 1$.
	\end{lem}
	
	\begin{proof}
		For any $f\in L^{p(\cdot)}(\mathbb{B}_n,dV)$ with $\|f\|_{p(\cdot)}\leq 1$. Let $K=\{w\in B(a,r):|f(w)|\leq 1\}$. Since $p_r^{-}:={\rm ess}\inf_{B(a,r)}p(z)>1$, by H${\rm\ddot{o}}$lder's inequality, we obtain
		\begin{equation*}
			\begin{split}
				&\left(\frac{1}{V(B(a,r))}\int_{B(a,r)}|f(w)|dV(w)\right)^{p(z)}\\
				&\quad \lesssim 1+\left(\frac{1}{V(B(a,r))}\int_{B(a,r)-K}|f(w)|^{p_r^-}dV(w)\right)^{p(z)/p_r^{-}}\\
				&\quad \lesssim 1+\left(\frac{1}{V(B(a,r))}\int_{B(a,r)}|f(w)|^{p(w)}dV(w)\right)^{p(z)/p_r^{-}}.
			\end{split}
		\end{equation*}
		By condition (i) in Lemma \ref{lemma1.1}, we know that
		\[\int_{B(a,r)}|f(w)|^{p(w)}dV(w)\leq \rho_{p(\cdot)}(f)\leq 1.\]
		And by \eqref{equa2.2} and Lemma \ref{lemma2.1}, we have $V(B(a,r))^{p(z)/p_r^{-}}\simeq V(B(a,r))$ when $z\in B(a,r)$. Hence we get
		\begin{equation*}
			\left(\frac{1}{V(B(a,r))}\int_{B(a,r)}|f(w)|dV(w)\right)^{p(z)}\lesssim 1+ \frac{1}{V(B(a,r))}\int_{B(a,r)}|f(w)|^{p(w)}dV(w)
		\end{equation*} 
		for all $a\in\mathbb{B}_n$ and $z\in B(a,r)$.
	\end{proof}
	
	Using Lemma \ref{lemma2.2}, we could get the following pointwise estimation for functions in $A^{p(\cdot)}(\mathbb{B}_n)$.
	
	\begin{lem}\label{lemma2.3}
		Let $p(\cdot)\in\mathcal{P}^{log}(\mathbb{B}_n)$. Then
		\[|f(z)|\lesssim \frac{\|f\|_{p(\cdot)}}{(1-|z|^2)^{(n+1)/p(z)}}\]
		for all $f\in A^{p(\cdot)}(\mathbb{B}_n)$ and $z\in\mathbb{B}_n$.
	\end{lem}
	
	\begin{proof}
		Let $f\in A^{p(\cdot)}(\mathbb{B}_n)$ with $\|f\|_{p(\cdot)}=1$. Then $\rho_{p(\cdot)}(f)=1$ by condition (i) in Lemma \ref{lemma1.1}. For any $z\in\mathbb{B}_n$ and $r>0$, by the sub-mean value property of $|f|$, we have
		\[|f(z)|\lesssim \frac{1}{V(B(z,r))}\int_{B(z,r)}|f(w)|dV(w).\]
		Applying Lemma \ref{lemma2.2} and \eqref{equa2.2}, we obtain
		\begin{equation}\label{equa2.6}
			|f(z)|^{p(z)}\lesssim \frac{1}{(1-|z|^2)^{n+1}}\int_{B(z,r)}|f(w)|^{p(w)}dV(w)+1\lesssim \frac{1}{(1-|z|^2)^{n+1}}.
		\end{equation}
		Now for a general $f\in A^{p(\cdot)}(\mathbb{B}_n)$, we consider $\frac{f}{\|f\|_{p(\cdot)}}$. A routine scaling argument yields that
		\[|f(z)|\lesssim \frac{1}{(1-|z|^2)^{(n+1)/p(z)}}\|f\|_{p(\cdot)}.\]
		Constant suppressed here is independent of $z$ and $f$.
	\end{proof}
	
	Recall the reproducing formula \eqref{equa1.1}, by the dual theory in $A^{p(\cdot)}(\mathbb{B}_n)$, Lemma \ref{lemma2.3} tells us that $\|K_z\|_{p'(\cdot)}\lesssim \frac{1}{(1-|z|^2)^{(n+1)/p(z)}}$ for all $z\in\mathbb{B}_n$, where $K_z$ is the kernel function in \eqref{equa1.2}.
	
	Let $N$ be a positive integer. Inspired by the form of the function $K_z$, for any $z\in\mathbb{B}_n$, let 
	\[F_{z,N}(w)=\frac{1}{(1-\langle w,z\rangle)^{N}},\quad w\in\mathbb{B}_n.\]
	It is easy to check that 
	\begin{equation}\label{equa2.7}
		\|F_{z,N}\|_{p(\cdot)}\simeq (1-|z|^2)^{\frac{n+1}{p(z)}-N}
	\end{equation}
	for any $z\in\mathbb{B}_n$, when $p(\cdot)\in\mathcal{P}^{log}(\mathbb{B}_n)$ and $N\geq n+1$. And constant suppressed is independent of $z$.
	
	For $1\leq j\leq n$, let $e_j=(0,\cdots,0,1,0,\cdots,0)$, where $1$ is on the $j$-th component. For any $a\in\mathbb{B}_n\backslash \{0\}$, choose an unitary transformation $U$ such that $Ua=|a|e_1$. Denote by $a^j=U^*(|a|e_j)$, $j=2,\cdots,n$.
	
	The following two lemmas derives from \cite{KhWm}, which are essential in the study of the difference of two weighted composition operators.
	
	\begin{lem}\label{lemma2.4}
		Let $r>0$. There exists $0<t_0<1$ such that
		\[d(z,w)\simeq \frac{1}{|1-\langle z,w\rangle|}\left(|\langle z-w,a\rangle|+\sqrt{1-|a|}\sum_{j=2}^{n}|\langle z-w,a^j\rangle|\right)\]
		for all $a\in\mathbb{B}_n$ with $t_0<|a|<1$, $z\in B(a,r)$ and $w\in\mathbb{B}_n$.
	\end{lem}
	
	\begin{proof}
		According to \cite[Lemma 2.1]{KhWm}, we know that
		\[d(z,w)\simeq \frac{1}{|1-\langle z,w\rangle|}\left(|z_1-w_1|+\sqrt{1-t}\sum_{j=2}^n|z_j-w_j|\right)\]
		for all $z\in B(te_1,r)$ with $t_0<t<1$ and $w\in\mathbb{B}_n$.
		
		Let $a\in\mathbb{B}_n$ with $|a|>t_0$, choose an unitary transformation $U$ such that $Ua=|a|e_1$. Note that $z\in B(a,r)$ if and only if $Uz\in B(|a|e_1,r)$. By the unitary invariance of the distance $d$, we get
		\begin{equation*}
			\begin{split}
				d(z,w)&=d(Uz,Uw)\\
				&\simeq \frac{1}{|1-\langle z,w\rangle|}\left(|\langle Uz-Uw,|a|e_1\rangle|+\sqrt{1-|a|}\sum_{j=2}^n|\langle Uz-Uw,|a|e_j|\right)\\
				&\simeq \frac{1}{|1-\langle z,w\rangle|}\left(|\langle z-w,a\rangle|+\sqrt{1-|a|}\sum_{j=2}^{n}|\langle z-w,a^j\rangle|\right).
			\end{split}
		\end{equation*}
		The proof is complete.
	\end{proof}
	
	\begin{lem}{\cite{KhWm}}\label{lemma2.5}
		Let $0<p<\infty$ and $0<s_1<s_2<1$. Then there exists a constant $C=C(s_1,s_2)>0$ such that
		\[|f(z)-f(w)|^p\leq C\frac{d(z,w)^p}{V(E(z,s_2))}\int_{E(z,s_2)}|f(\zeta)|^pdV(\zeta)\]
		for all $z\in\mathbb{B}_n$, $w\in E(z,s_1)$ and $f\in H(\mathbb{B}_n)$.
	\end{lem}
	
	The geometric characterizations of Carleson measures for $A^{p(\cdot)}(\mathbb{B}_n)$ have been obtained in \cite{Da}. We state the results as follows.
	
	\begin{lem}\label{lemma2.06}
		Let $\mu$ be a positive Borel measure on $\mathbb{B}_n$ and $p(\cdot)\in\mathcal{P}^{log}(\mathbb{B}_n)$. Then
		\begin{itemize}
			\item [(i)] $\mu$ is a Carleson measure for $A^{p(\cdot)}(\mathbb{B}_n)$ if and only if 
			\[\sup_{a\in\mathbb{B}_n}\frac{\mu(B(a,r))}{(1-|a|^2)^{n+1}}<\infty\]
			for some (or any) $r>0$.
			\item[(ii)] $\mu$ is a compact Carleson measure for $A^{p(\cdot)}(\mathbb{B}_n)$ if and only if 
			\[\lim_{|a|\to 1}\frac{\mu(B(a,r))}{(1-|a|^2)^{n+1}}=0\]
			for some (or any) $r>0$.	
		\end{itemize}
	\end{lem}
	
	We end this section with the following criteria for the compactness of the operators $T_{\mu}^{\beta}$ and $W_{u,\varphi}$, which follows easily from the fact that $A^{p(\cdot)}(\mathbb{B}_n)$ is reflexive when $1<p^-\leq p^+<\infty$ and $\{f_j\}$ converges weakly if and only if it is bounded and converges uniformly on compact subsets of $\mathbb{B}_n$. See \cite[Proposition 3.11]{CcMb} for the case of constant variable.
	
	\begin{lem}\label{lemma2.6}
		Let $p(\cdot)\in \mathcal{P}^{log}(\mathbb{B}_n)$ and $T=T_{\mu}^{\beta}$ or $W_{u,\varphi}$. $T$ is compact on $A^{p(\cdot)}(\mathbb{B}_n)$ if and only if $\|Tf_j\|\to 0$ (or equivalently $\rho_{p(\cdot)}(Tf_j)\to 0$) for any bounded sequence $\{f_j\}$ in $A^{p(\cdot)}(\mathbb{B}_n)$ that converges to 0 uniformly on compact subsets of $\mathbb{B}_n$.
	\end{lem}

	\section{Carleson measure and Toeplitz operator}
	
	In this section, we investigate the boundedness and compactness of the Toeplitz operator $T_{\mu}^{\beta}$ on $A^{p(\cdot)}(\mathbb{B}_n)$. Before that, we need the following well-known covering lemma of $\mathbb{B}_n$, see \cite[Theorem 2.23]{Zk}.
	
	\begin{lem}\label{lemma3.1}
		There exists a positive integer $N_0$ such that for any $r>0$ we can find a sequence $\{a_k\}$ in $\mathbb{B}_n$ with the following properties:
		\begin{itemize}
			\item[(i)] $\mathbb{B}_n=\cup_kB(a_k,r)$.
			\item[(ii)] The sets $B(a_k,\frac{r}{4})$ are mutually disjoint.
			\item[(iii)] Each point $z\in\mathbb{B}_n$ belongs to at most $N_0$ of the sets $B(a_k,4r)$.
		\end{itemize}
		The sequence $\{a_k\}$ above is called an $r$-lattice in the Bergman metric.
	\end{lem}
	
	We are now ready to characterize the boundedness and compactness of the Toeplitz operator $T_{\mu}^{\beta}$ on $A^{p(\cdot)}(\mathbb{B}_n)$.
	
	\begin{thm}\label{theorem3.4}
		Let $\mu$ be a positive Borel measure on $\mathbb{B}_n$, $p(\cdot)\in \mathcal{P}^{log}(\mathbb{B}_n)$ and $\beta\geq 0$. $T_{\mu}^{\beta}$ is bounded on $\mathbb{B}_n$ if and only if 
		\[\sup_{a\in\mathbb{B}_n}\frac{\mu(B(a,r))}{(1-|a|^2)^{n+1+\beta}}<\infty\]
		for some (or any) $r>0$.
	\end{thm}
	
	\begin{proof}
		{\it Necessity}: For any $a\in\mathbb{B}_n$ and fixed $N\geq n+1$, let 
		\[f_{a,N+\beta}(z)=\frac{(1-|a|^2)^{N+\beta-\frac{n+1}{p(a)}}}{(1-\langle z,a\rangle)^{N+\beta}},\quad z\in\mathbb{B}_n,\]
		and
		\[g_{a,\beta}(z)=h_{a,\beta}(z)(1-|z|^2)^{\beta},\quad z\in\mathbb{B}_n,\]
		where 
		\[h_{a,\beta}(z)=\frac{(1-|a|^2)^{N-\frac{n+1}{p'(a)}}}{(1-\langle z,a\rangle)^{N+\beta}}.\]
		By \eqref{equa2.7}, we know that $\sup_{a\in\mathbb{B}_n}\|f_{a,N+\beta}\|_{p(\cdot)}<\infty$ and $\sup_{a\in\mathbb{B}_n}\|g_{a,\beta}\|_{p'(\cdot)}<\infty$. It follows that $h_{a,\beta}\in A_{\beta}^1(\mathbb{B}_n)$.
		
		By the dual theory in $A^{p(\cdot)}(\mathbb{B}_n)$ and the reproducing formula in $A_{\beta}^1(\mathbb{B}_n)$, we use Fubini's Theorem and \eqref{equa2.3} to obtain
		\begin{equation}\label{equa3.4}
			\begin{split}
				\|T_{\mu}^{\beta}f_{a,N+\beta}\|_{p(\cdot)}&\gtrsim \left|\int_{\mathbb{B}_n}(T_{\mu}^{\beta}f_{a,N+\beta})(z)\overline{g_{a,\beta}(z)}dV(z)\right|\\
				&=\int_{\mathbb{B}_n}f_{a,N+\beta}(w)\int_{\mathbb{B}_n}\overline{h_{a,\beta}(z)}\frac{(1-|z|^2)^{\beta}}{(1-\langle z,w\rangle)^{n+1+\beta}}dV(z)d\mu(w)\\
				&=C_{\beta}\int_{\mathbb{B}_n}f_{a,N+\beta}(w)\overline{h_{a,\beta}(w)}d\mu(w)\\
				&\geq C_{\beta}\int_{B(a,r)}\frac{(1-|a|^2)^{2N+\beta-(n+1)}}{|1-\langle a,w\rangle|^{2(N+\beta)}}d\mu(w)\\
				&\simeq \frac{\mu(B(a,r))}{(1-|a|^2)^{n+1+\beta}}.
			\end{split}
		\end{equation}
		Here $\frac{1}{C_{\beta}}=\int_{\mathbb{B}_n}(1-|z|^2)^{\beta}dV(z)$. Consequently, the boundedness of $T_{\mu}^{\beta}$ on $A^{p(\cdot)}(\mathbb{B}_n)$ implies that 
		\[\sup_{a\in\mathbb{B}_n}\frac{\mu(B(a,r))}{(1-|a|^2)^{n+1+\beta}}<\infty.\]
		
		{\it Sufficiency}: Assume $\hat{\mu}_{r,\beta}:=\frac{\mu(B(a,r))}{(1-|a|^2)^{n+1+\beta}}<\infty$ for some $r>0$. Let $\{a_k\}$ be an $r$-lattice in the Bergman metric. For any $f\in A^{p(\cdot)}(\mathbb{B}_n)$, according to Lemma \ref{lemma3.1} and the sub-mean value property of $|f|$, we obtain
		\begin{equation}\label{equa3.5}
			\begin{split}
				&|T_{\mu}^{\beta}f(z)|\\
				&\quad \leq \sum_{k=1}^{\infty}\int_{B(a_k,r)}\frac{|f(w)|}{|1-\langle w,z\rangle|^{n+1+\beta}}d\mu(w)\\
				&\quad \lesssim \sum_{k=1}^{\infty}\int_{B(a_k,r)}\left(\frac{1}{(1-|w|^2)^{n+1}}\int_{B(w,r)}\frac{|f(\zeta)|}{|1-\langle z,\zeta\rangle|^{n+1+\beta}}dV(\zeta)\right)d\mu(w)\\
				&\quad \lesssim \sum_{k=1}^{\infty}\int_{B(a_k,r)}\frac{1}{(1-|w|^2)^{n+1+\beta}}\int_{B(a_k,2r)}\frac{|f(\zeta)|}{|1-\langle z,\zeta\rangle|^{n+1}}dV(\zeta)d\mu(w)\\
				&\quad\lesssim N_0\hat{\mu}_{r,\beta}\int_{\mathbb{B}_n}\frac{|f(\zeta)|}{|1-\langle z,\zeta\rangle|^{n+1}}dV(\zeta)\lesssim P^*(|f|)(z).
			\end{split}
		\end{equation}
		Here, equations \eqref{equa2.3} and \eqref{equa2.4} are used in the above estimation. Consequently, the boundedness of $T_{\mu}^{\beta}$ follows from the boundedness of $\hat{P}$ on $L^{p(\cdot)}(\mathbb{B}_n,dV)$.
	\end{proof}
	
	\begin{thm}\label{theorem3.5}
		Let $\mu$ be a positive Borel measure on $\mathbb{B}_n$, $p(\cdot)\in\mathcal{P}^{log}(\mathbb{B}_n)$ and $\beta\geq 0$. $T_{\mu}^{\beta}$ is compact on $\mathbb{B}_n$ if and only if 
		\[\lim_{|a|\to 1}\frac{\mu(B(a,r))}{(1-|a|^2)^{n+1+\beta}}=0\]
		for some (or any) $r>0$.
	\end{thm}
	
	\begin{proof}
		{\it Necessity}: Assume $T_{\mu}^{\beta}$ is compact on $A^{p(\cdot)}(\mathbb{B}_n)$. Since $f_{a,N+\beta}$ is bounded in $A^{p(\cdot)}(\mathbb{B}_n)$ and converges to 0 uniformly on compact subsets of $\mathbb{B}_n$, we combine Lemma \ref{lemma2.6} and \eqref{equa3.4} to obtain
		\[\lim_{|a|\to 1}\frac{\mu(B(a,r))}{(1-|a|^2)^{n+1+\beta}}\lesssim \lim_{|a|\to 1}\|T_{\mu}^{\beta}f_{a,N+\beta}\|_{p(\cdot)}=0.\]
		
		{\it Sufficiency}: Let $\{a_k\}$ be an $r$-lattice in the Bergman metric. For any $\varepsilon>0$, choose $k_0\in\mathbb{N}$ such that
		\[\frac{\mu(B(a_k,r))}{(1-|a_k|^2)^{n+1+\beta}}<\varepsilon\]
		whenever $k>k_0$. Let $\{f_j\}$ be any bounded sequence in $A^{p(\cdot)}(\mathbb{B}_n)$ that converges to 0 uniformly on compact subsets of $\mathbb{B}_n$. Through a similar argument as in \eqref{equa3.5}, we obtain
		\[|T_{\mu}^{\beta}f_j(z)|\lesssim \hat{\mu}_{r,\beta}\sum_{k=1}^{k_0}\int_{B(a_k,2r)}\frac{|f_j(\zeta)|}{|1-\langle z,\zeta\rangle|^{n+1}}dV(\zeta)+N_0\varepsilon P^*(|f_j|)(z).\]
		Choose $J\in\mathbb{N}$ such that
		\[\sum_{k=1}^{k_0}\int_{B(a_k,2r)}\frac{|f_j(\zeta)|}{|1-\langle z,\zeta\rangle|^{n+1}}dV(\zeta)\lesssim \sum_{k=1}^{k_0}\int_{B(a_k,2r)}|f_{j}(\zeta)|dV(\zeta)<\varepsilon\]
		whenever $j>J$.
		
		Consequently, for $j>J$, we have
		\[\rho_{p(\cdot)}(T_{\mu}^{\beta}f_j)\lesssim (\hat{\mu}_{r,\beta}+1)^{p^+}\varepsilon+N_0^{p^+}\varepsilon \left(\rho_{p(\cdot)}(P^*|f_j|)\right).\]
		Then by Lemma \ref{lemma2.6}, we conclude that $T_{\mu}^{\beta}$ is compact on $A^{p(\cdot)}(\mathbb{B}_n)$.
	\end{proof}
	
	\section{Difference of weighted composition operators}
	
	In this section, we describe the properties of $C_{u,\varphi}$, and characterize bounded and compact differences $C_{u,\varphi}-C_{v,\psi}$ on $A^{p(\cdot)}(\mathbb{B}_n)$.
	
	\subsection{$C_{u,\varphi}$ on $A^{p(\cdot)}(\mathbb{B}_n)$}
	
	Let $p(\cdot)\in\mathcal{P}^{log}(\mathbb{B}_n)$, $u\in H(\mathbb{B}_n)$ and $\varphi\in S(\mathbb{B}_n)$. We define a function $\omega_{\varphi}$ on $\mathbb{B}_n$ by
	\begin{equation}
		\omega_{\varphi}(z)=\left(\frac{1}{1-|\varphi(z)|^2}\right)^{(n+1)\frac{p(z)-p(\varphi(z))}{p(\varphi(z))}},\quad z\in\mathbb{B}_n,
	\end{equation}
	and define two weighted pull-back measures on $\mathbb{B}_n$ as follows:
	\[\mu_{u,\varphi}(E)=\int_{\varphi^{-1}(E)}|u(z)|^{p(z)}\omega_{\varphi}(z)dV(z);\]
	\[\mu_{u,\varphi}^{(1)}(E)=\int_{\varphi^{-1}(E)}|u(z)|^{p(z)}(\omega_{\varphi}(z)+1)dV(z),\]
	where $E$ is any Borel subset of $\mathbb{B}_n$.
	
	We first present some necessary conditions for the boundedness and compactness of $C_{u,\varphi}$ on $A^{p(\cdot)}(\mathbb{B}_n)$.
	
	\begin{prop}\label{proposition4.1}
		Suppose $p(\cdot)\in\mathcal{P}^{log}(\mathbb{B}_n)$, $u\in H(\mathbb{B}_n)$ and $\varphi\in S(\mathbb{B}_n)$. If $C_{u,\varphi}$ is bounded on $A^{p(\cdot)}(\mathbb{B}_n)$, then $u\in A^{p(\cdot)}(\mathbb{B}_n)$ and 
		\begin{itemize}
			\item[(i)]
			\begin{equation}\label{equa4.1}
				\sup_{z\in\mathbb{B}_n}|u(z)|\frac{(1-|z|^2)^{(n+1)/p(z)}}{(1-|\varphi(z)|^2)^{(n+1)/p(\varphi(z))}}<\infty.
			\end{equation}
			\item[(ii)] The measure $\mu_{u,\varphi}$ is a Carleson measure for $A^{p(\cdot)}(\mathbb{B}_n)$.
		\end{itemize}
	\end{prop}
	
	\begin{proof}
		If $C_{u,\varphi}$ is bounded on $A^{p(\cdot)}(\mathbb{B}_n)$, clearly, $u=C_{u,\varphi}1\in A^{p(\cdot)}(\mathbb{B}_n)$. For any $a\in\mathbb{B}_n$ and fixed $N\geq n+1$, recall that
		\[f_{a,N}(z)=\frac{(1-|a|^2)^{N-\frac{n+1}{p(a)}}}{(1-\langle z,a\rangle)^{N}},\quad z\in\mathbb{B}_n.\]
		The boundedness of $C_{u,\varphi}$ implies that
		\[\sup_{a\in\mathbb{B}_n}\|C_{u,\varphi}f_{a,N}\|_{p(\cdot)}\leq C \sup_{a\in\mathbb{B}_n}\|f_{a,N}\|_{p(\cdot)}<\infty.\]
		On the other hand, for any $z\in\mathbb{B}_n$, by Lemma \ref{lemma2.3}, we obtain
		\begin{equation}\label{equa4.3}
			\begin{split}
				\|C_{u,\varphi}f_{\varphi(z),N}\|_{p(\cdot)}&\gtrsim (1-|z|^2)^{(n+1)/p(z)}\left|(C_{u,\varphi}f_{\varphi(z),N})(z)\right|\\
				&=|u(z)|\frac{(1-|z|^2)^{(n+1)/p(z)}}{(1-|\varphi(z)|^2)^{(n+1)/p(\varphi(z))}}.
			\end{split}
		\end{equation}
		Therefore,
		\[\sup_{z\in\mathbb{B}_n}|u(z)|\frac{(1-|z|^2)^{(n+1)/p(z)}}{(1-|\varphi(z)|^2)^{(n+1)/p(\varphi(z))}}<\infty.\]
		
		Moreover, condition (i) in Lemma \ref{lemma1.1} tells us that
		\[\sup_{a\in\mathbb{B}_n}\rho_{p(\cdot)}(C_{u,\varphi}f_{a,N})\leq \left(C\sup_{a\in\mathbb{B}_n}\|f_{a,N}\|_{p(\cdot)}+1\right)^{p^+}<\infty.\]
		For any $r>0$, by \eqref{equa2.3} and Lemma \ref{lemma2.1}, we obtain
		\begin{equation}\label{equa4.4}
			\begin{split}
				\rho_{p(\cdot)}(C_{u,\varphi}f_{a,N})&\geq\int_{\varphi^{-1}(B(a,r))}|u(z)f_{a,N}(\varphi(z))|^{p(z)}dV(z)\\
				&\gtrsim\int_{\varphi^{-1}(B(a,r))}|u(z)|^{p(z)}\left(\frac{1}{1-|a|^2}\right)^{\frac{(n+1)p(z)}{p(a)}}dV(z)\\
				&\simeq \frac{\int_{\varphi^{-1}(B(a,r))}|u(z)|^{p(z)}\omega_{\varphi}(z)dV(z)}{(1-|a|^2)^{n+1}}.
			\end{split}
		\end{equation}
		Therefore, 
		\[\sup_{a\in\mathbb{B}_n}\frac{\int_{\varphi^{-1}(B(a,r))}|u(z)|^{p(z)}\omega_{\varphi}(z)dV(z)}{(1-|a|^2)^{n+1}}<\infty.\]
		This shows that $\mu_{u,\varphi}$ is a Carleson measure for $A^{p(\cdot)}(\mathbb{B}_n)$ by Lemma \ref{lemma2.06}.
	\end{proof}
	
	\begin{ex}\label{example4.2}
		Let $u_0(z)=z_1+z_2+\cdots+z_n$ and $\varphi_0(z)=-z$, it is obvious that $C_{u_0,\varphi_0}$ is bounded on $A^{p}(\mathbb{B}_n)$ for any $0<p<\infty$. However, if 
		\[p_0(z)=(n+3)+{\rm Re} (z_1+\cdots+z_n),\quad z=(z_1,\cdots,z_n)\in\mathbb{B}_n,\]
		then $p_0(\cdot)\in\mathcal{P}^{log}(\mathbb{B}_n)$ and
		\begin{equation*}
			\begin{split}
				&\sup_{z\in\mathbb{B}_n}|u(z)|\frac{(1-|z|^2)^{(n+1)/p(z)}}{(1-|\varphi(z)|^2)^{(n+1)/p(\varphi(z))}}\\
				&\quad \geq\lim_{\substack{x_j\to\frac{1}{\sqrt{n}}, y_j\to 0\\ 1\leq j\leq n}}|x_1+\cdots+x_n|(1-x_1^2-\cdots-x_n^2)^{\frac{n+1}{n+3+\sqrt{n}}-\frac{n+1}{n+3-\sqrt{n}}}=\infty.
			\end{split}
		\end{equation*}
		This shows that $C_{u_0,\varphi_0}$ is unbounded on $A^{p_0(\cdot)}(\mathbb{B}_n)$.
	\end{ex}
	
	\begin{prop}\label{proposition4.2}
		Suppose $p(\cdot)\in\mathcal{P}^{log}(\mathbb{B}_n)$, $u\in H(\mathbb{B}_n)$ and $\varphi\in S(\mathbb{B}_n)$. If $C_{u,\varphi}$ is compact on $A^{p(\cdot)}(\mathbb{B}_n)$, then
		\begin{itemize}
			\item [(i)]
			\[\lim_{|\varphi(z)|\to 1}|u(z)|^{p(z)}\frac{(1-|z|^2)^{(n+1)/p(z)}}{(1-|\varphi(z)|^2)^{(n+1)/p(\varphi(z))}}=0.\]
			\item[(ii)] The measure $\mu_{u,\varphi}$ is a compact Carleson measure for $A^{p(\cdot)}(\mathbb{B}_n)$.
		\end{itemize}
	\end{prop}
	
	\begin{proof}
		Since $f_{a,N}$ is bounded in $A^{p(\cdot)}(\mathbb{B}_n)$ and converges to 0 uniformly on compact subsets of $\mathbb{B}_n$ as $|a|\to 1$, by Lemma \ref{lemma2.6}, the compactness of $C_{u,\varphi}$ implies that
		\begin{equation*}
			\lim_{|a|\to 1}\|C_{u,\varphi}f_{a,N}\|_{p(\cdot)}=\lim_{|a|\to\infty}\rho_{p(\cdot)}(C_{u,\varphi}f_{a,N})=0.
		\end{equation*}
		This, together with \eqref{equa4.3}, shows that
		\[\lim_{|\varphi(z)|\to 1}|u(z)|\frac{(1-|z|^2)^{(n+1)/p(z)}}{(1-|\varphi(z)|^2)^{(n+1)/p(\varphi(z))}}=0.\]
		And by \eqref{equa4.4}, we get
		\begin{equation*}
			\lim_{|a|\to 1}\frac{\int_{\varphi^{-1}(B(a,r))}|u(z)|^{p(z)}\omega_{\varphi}(z)dV(z)}{(1-|a|^2)^{n+1}}=0,
		\end{equation*}
		which shows that $\mu_{u,\varphi}$ is a compact Carleson measure for $A^{p(\cdot)}(\mathbb{B}_n)$ by Lemma \ref{lemma2.06}.
	\end{proof}
	
	Now we give a sufficient condition for the boundedness of $C_{u,\varphi}$ on $A^{p(\cdot)}(\mathbb{B}_n)$.
	
	\begin{prop}\label{proposition4.3}
		Let $p(\cdot)\in\mathcal{P}^{log}(\mathbb{B}_n)$, $u\in A^{p(\cdot)}(\mathbb{B}_n)$ and $\varphi\in S(\mathbb{B}_n)$. If $\mu_{u,\varphi}^{(1)}$ is a Carleson measure for $A^{p(\cdot)}(\mathbb{B}_n)$, then $C_{u,\varphi}$ is bounded on $A^{p(\cdot)}(\mathbb{B}_n)$.
	\end{prop}
	
	\begin{proof}
		Assume $\mu_{u,\varphi}^{(1)}$ is a Carleson measure for $A^{p(\cdot)}(\mathbb{B}_n)$, then there exists a constant $C>0$ such that $\|f\|_{p(\cdot),\mu_{u,\varphi}^{(1)}}\leq C\|f\|_{p(\cdot)}$ for any $f\in A^{p(\cdot)}(\mathbb{B}_n)$. Let $f\in A^{p(\cdot)}(\mathbb{B}_n)$ with $\|f\|_{p(\cdot)}=1$. By Lemma \ref{lemma2.3} and condition (i) in Lemma \ref{lemma1.1}, we obtain
		\begin{equation*}
			\begin{split}
				\rho_{p(\cdot)}(C_{u,\varphi}f)&=\int_{\mathbb{B}_n}|u(z)|^{p(z)}|f(\varphi(z))|^{p(z)}dV(z)\\
				&\leq \int_{\mathbb{B}_n}|u(z)|^{p(z)}dV(z)+\int_{\mathbb{B}_n}|f(\varphi(z))|^{p(\varphi(z))}|u(z)|^{p(z)}(\omega_{\varphi}(z)+1)dV(z)\\
				&=\rho_{p(\cdot)}(u)+\rho_{p(\cdot),\mu_{u,\varphi}^{(1)}}(f)\\
				&\leq \rho_{p(\cdot)}(u)+(C+1)^{p^+}.
			\end{split}
		\end{equation*}
		Then by condition (ii) in Lemma \ref{lemma1.1}, we obtain
		\[\|C_{u,\varphi}f\|_{p(\cdot)}\leq \rho_{p(\cdot)}(C_{u,\varphi}f)+1\leq \rho_{p(\cdot)}(u)+(C+1)^{p^+}+1.\]
		For a general $f\in A^{p(\cdot)}(\mathbb{B}_n)$, considering $\frac{f}{\|f\|_{p(\cdot)}}$, a routine scaling argument yields that 
		\[\|C_{u,\varphi}f\|_{p(\cdot)}\leq \left(\rho_{p(\cdot)}(u)+(C+1)^{p^+}+1\right)\|f\|_{p(\cdot)}.\]
		This shows the boundedness of $C_{u,\varphi}$ on $A^{p(\cdot)}(\mathbb{B}_n)$.
	\end{proof}
	
	\begin{thm}\label{theorem4.4}
		Let $p(\cdot)\in\mathcal{P}^{log}(\mathbb{B}_n)$, $u\in H(\mathbb{B}_n)$ and $\varphi\in S(\mathbb{B}_n)$. Suppose $p(z)\geq p(\varphi(z))$ a.e. Then $C_{u,\varphi}$ is bounded (compact, resp.) if and only if $u\in A^{p(\cdot)}(\mathbb{B}_n)$ and $\mu_{u,\varphi}$ is a (compact, resp.) Carleson measure for $A^{p(\cdot)}(\mathbb{B}_n)$.
	\end{thm}
	
	\begin{proof}
		The necessity has been proved in Proposition \ref{proposition4.1} and Proposition \ref{proposition4.2}. And the sufficiency for the boundedness part follows from Proposition \ref{proposition4.3} since $\omega_{\varphi}(z)\simeq \omega_{\varphi}(z)+1$ when $p(z)\geq p(\varphi(z))$. It remains to prove the sufficiency for the compactness part.
		
		Assume $\mu_{u,\varphi}$ is a compact Carleson measure for $A^{p(\cdot)}(\mathbb{B}_n)$. Let $\{f_j\}$ be any bounded sequence in $A^{p(\cdot)}(\mathbb{B}_n)$ that converges to 0 uniformly on compact subsets of $\mathbb{B}_n$. By Lemma \ref{lemma2.3}, we obtain
		\begin{equation*}
			\begin{split}
				\rho_{p(\cdot)}(C_{u,\varphi}f_j)&=\int_{\mathbb{B}_n}|u(z)|^{p(z)}|f_j(\varphi(z))|^{p(z)}dV(z)\\
				&\lesssim \int_{\mathbb{B}_n}|f_j(\varphi(z))|^{p(\varphi(z))}|u(z)|^{p(z)}\omega_{\varphi}(z)dV(z)\\
				&=\rho_{p(\cdot),\mu_{u,\varphi}}(f_j)\to 0.
			\end{split}
		\end{equation*}
		Then it follows from Lemma \ref{lemma2.6} that $C_{u,\varphi}$ is compact on $A^{p(\cdot)}(\mathbb{B}_n)$.
	\end{proof}
	
	\subsection{$C_{u,\varphi}-C_{v,\psi}$ on $A^{p(\cdot)}(\mathbb{B}_n)$}
	
	Let $p(\cdot)\in\mathcal{P}^{log}(\mathbb{B}_n)$, $u,v\in H(\mathbb{B}_n)$ and $\varphi,\psi\in S(\mathbb{B}_n)$. We put 
	\[d(z)=d(\varphi(z),\psi(z)),\quad z\in\mathbb{B}_n.\]
	To characterize bounded and compact differences of two weighted composition operators on $A^{p(\cdot)}(\mathbb{B}_n)$, we introduce some weighted pull-back measures on $\mathbb{B}_n$ associated with $\varphi$, $\psi$, $u$, $v$ and $p(\cdot)$ as follows. For any Borel set $E\subset \mathbb{B}_n$, let 
	\[\mu_{u,\varphi,d}(E)=\int_{\varphi^{-1}(E)}|u(z)|^{p(z)}d(z)^{p(z)}\omega_{\varphi}(z)dV(z),\]
	\[\mu_{u,\varphi,d}^{(1)}(E)=\int_{\varphi^{-1}(E)}|u(z)|^{p(z)}\left[d(z)^{p(z)}\omega_{\varphi}(z)+1\right]dV(z),\]
	\[\lambda_{\varphi,\alpha}(E)=\int_{\varphi^{-1}(E)}|u(z)-v(z)|^{p(z)}(1-d(z))^{\alpha}\omega_{\varphi}(z)dV(z),\]
	and
	\[\lambda_{\varphi,\alpha}^{(1)}(E)=\int_{\varphi^{-1}(E)}|u(z)-v(z)|^{p(z)}(1-d(z))^{\alpha}(\omega_{\varphi}(z)+1)dV(z),\]
	where $\omega_{\varphi}$ is the weight function defined in \eqref{equa4.1} and $\alpha\geq (n+1)p^+$. Also, we can define $\mu_{v,\psi,d}$, $\mu_{v,\psi,d}^{(1)}$, $\lambda_{\psi,\alpha}$ and $\lambda_{\psi,\alpha}^{(1)}$ accordingly.
	
	\begin{thm}\label{theorem4.5}
		Let $p(\cdot)\in\mathcal{P}^{log}(\mathbb{B}_n)$, $u,v\in H(\mathbb{B}_n)$ and $\varphi,\psi\in S(\mathbb{B}_n)$. If $C_{u,\varphi}-C_{v,\psi}$ is bounded (compact, resp.) on $A^{p(\cdot)}(\mathbb{B}_n)$, then the measures $\mu_{u,\varphi,d}$, $\mu_{v,\psi,d}$, $\lambda_{\varphi,\alpha}$ and $\lambda_{\psi,\alpha}$ are (compact, resp.) Carleson measures for $A^{p(\cdot)}(\mathbb{B}_n)$.
	\end{thm}
	
	\begin{proof}
		{\it Boundedness}: For any $a\in\mathbb{B}_n$ and $N\geq n+1$, recall that 
		\[f_{a,N}(z)=\frac{(1-|a|^2)^{N-\frac{n+1}{p(a)}}}{(1-\langle z,a\rangle)^{N}},\quad z\in\mathbb{B}_n.\]
		Assume $C_{u,\varphi}-C_{v,\psi}$ is bounded on $A^{p(\cdot)}(\mathbb{B}_n)$. It follows from condition (i) in Lemma \ref{lemma1.1} that
		\[\sup_{a\in\mathbb{B}_n}\rho_{p(\cdot)}\left((C_{u,\varphi}-C_{v,\psi})f_{a,N}\right)\leq \left(C\sup_{a\in\mathbb{B}_n}\|f_{a,N}\|_{p(\cdot)}+1\right)^{p^+}<\infty.\]
		
		For fixed $s\in (0,1)$, by \eqref{equa2.3} and Lemma \ref{lemma2.1}, we have
		\begin{equation*}
			\begin{split}
				&\rho_{p(\cdot)}\left((C_{u,\varphi}-C_{v,\psi})f_{a,N}\right)\\
				&\quad \gtrsim \int_{\varphi^{-1}(E(a,\frac{s}{2}))}\left|\frac{u(z)}{(1-\langle\varphi(z),a\rangle)^N}-\frac{v(z)}{(1-\langle \psi(z),a\rangle)^{N}}\right|^{p(z)}\\
				&\quad\quad\quad \times(1-|a|^2)^{\left(N-\frac{n+1}{p(a)}\right)p(z)}dV(z)\\
				&\quad \simeq\frac{\int_{\varphi^{-1}(E(a,\frac{s}{2}))}\left|u(z)-v(z)\left(\frac{1-\langle \varphi(z),a\rangle}{1-\langle \psi(z),a\rangle}\right)^N\right|^{p(z)}\omega_{\varphi}(z)dV(z)}{(1-|a|^2)^{n+1}}.
			\end{split}
		\end{equation*}
		Hence,
		\begin{equation}\label{equa4.5}
			\sup_{a\in\mathbb{B}_n}\frac{\int_{\varphi^{-1}(E(a,\frac{s}{2}))}\left|u(z)-v(z)\left(\frac{1-\langle \varphi(z),a\rangle}{1-\langle \psi(z),a\rangle}\right)^N\right|^{p(z)}\omega_{\varphi}(z)dV(z)}{(1-|a|^2)^{n+1}}<\infty.
		\end{equation}
		Similarly,
		\begin{equation}\label{equa4.6}
			\sup_{a\in\mathbb{B}_n}\frac{\int_{\varphi^{-1}(E(a,\frac{s}{2}))}\left|u(z)-v(z)\left(\frac{1-\langle \varphi(z),a\rangle}{1-\langle \psi(z),a\rangle}\right)^{N+1}\right|^{p(z)}\omega_{\varphi}(z)dV(z)}{(1-|a|^2)^{n+1}}<\infty.
		\end{equation}
		Note that $\left|\frac{1-\langle \varphi(z),a\rangle}{1-\langle \psi(z),a\rangle}\right|\lesssim 1$ whenever $\varphi(z)\in E(a,\frac{s}{2})$. It follows from \eqref{equa4.5} that
		\begin{equation}\label{equa4.7}
			\sup_{a\in\mathbb{B}_n}\frac{\int_{\varphi^{-1}(E(a,\frac{s}{2}))}\left|u(z)\frac{1-\langle \varphi(z),a\rangle}{1-\langle\psi(z),a\rangle}-v(z)\left(\frac{1-\langle \varphi(z),a\rangle}{1-\langle \psi(z),a\rangle}\right)^N\right|^{p(z)}\omega_{\varphi}(z)dV(z)}{(1-|a|^2)^{n+1}}<\infty.
		\end{equation}
		Adding \eqref{equa4.6} and \eqref{equa4.7}, and by the triangle inequality, we obtain
		\begin{equation}\label{equa4.8}
			\sup_{a\in\mathbb{B}_n}\int_{\varphi^{-1}(E(a,\frac{s}{2}))}\left|\frac{\langle \varphi(z)-\psi(z),a\rangle}{1-\langle \psi(z),a\rangle}\right|^{p(z)}\frac{|u(z)|^{p(z)}\omega_{\varphi}(z)}{(1-|a|^2)^{n+1}}dV(z)<\infty.
		\end{equation}
		Choose $s_0\in (0,1)$ such that $E(a,\frac{s}{2})\subset E(b,s)$ whenever $d(a,b)<s_0$. Let
		\[b^j=a+s_0\sqrt{1-|a|^2}a^j,\quad j=2,\cdots,n.\]
		Then $d(a,b^j)<s_0$. Applying a similar argument as above and by \eqref{equa2.3} and \eqref{equa2.4}, we obtain
		\begin{equation}\label{equa4.9}
			\begin{split}
				&\sup_{a\in\mathbb{B}_n}\int_{\varphi^{-1}(E(a,\frac{s}{2}))}\left|\frac{\langle \varphi(z)-\psi(z),b^j\rangle}{1-\langle \psi(z),a\rangle}\right|^{p(z)}\frac{|u(z)|^{p(z)}\omega_{\varphi}(z)}{(1-|a|^2)^{n+1}}dV(z)\\
				&\quad \lesssim \sup_{a\in\mathbb{B}_n}\int_{\varphi^{-1}(E(b^j,s))}\left|\frac{\langle \varphi(z)-\psi(z),b^j\rangle}{1-\langle \psi(z),b^j\rangle}\right|^{p(z)}\frac{|u(z)|^{p(z)}\omega_{\varphi}(z)}{(1-|b^j|^2)^{n+1}}dV(z)\\
				&\quad<\infty.
			\end{split}
		\end{equation}
		Combining \eqref{equa4.8} and \eqref{equa4.9}, we obtain
		\begin{equation}\label{equa4.10}
			\sup_{a\in\mathbb{B}_n}\int_{\varphi^{-1}(E(a,\frac{s}{2}))}\left|\frac{\langle \varphi(z)-\psi(z),\sqrt{1-|a|^2}a^j\rangle}{1-\langle \psi(z),a\rangle}\right|^{p(z)}\frac{|u(z)|^{p(z)}\omega_{\varphi}(z)}{(1-|a|^2)^{n+1}}dV(z)<\infty
		\end{equation}
		for $j=2,\cdots,n$.
		
		Therefore, when $t_0<|a|<1$, using \eqref{equa2.4} and Lemma \ref{lemma2.4}, we combine \eqref{equa4.8} and \eqref{equa4.10} to obtain
		\begin{equation}\label{equa4.11}
			\begin{split}
				\sup_{|a|>t_0}\frac{\int_{\varphi^{-1}(E(a,\frac{s}{2}))}|u(z)|^{p(z)}d(z)^{p(z)}\omega_{\varphi}(z)dV(z)}{(1-|a|^2)^{n+1}}<\infty.
			\end{split}
		\end{equation}
		Similarly, 
		\begin{equation}\label{equa4.12}
			\sup_{|a|>t_0}\frac{\int_{\psi^{-1}(E(a,\frac{s}{2}))}|v(z)|^{p(z)}d(z)^{p(z)}\omega_{\psi}(z)dV(z)}{(1-|a|^2)^{n+1}}<\infty.
		\end{equation}
		On the other hand, when $|a|<t_0$, we have
		\begin{equation}\label{equa4.13}
			\begin{split}
				&\frac{\mu_{u,\varphi,d}(E(a,\frac{s}{2}))+\mu_{v,\psi,d}(E(a,\frac{s}{2}))}{(1-|a|^2)^{n+1}}\\
				&\quad\lesssim \int_{\varphi^{-1}(E(a,\frac{s}{2}))\cup\psi^{-1}(E(a,\frac{s}{2}))}\left(|u(z)|+|v(z)|\right)^{p(z)}|\varphi(z)-\psi(z)|^{p(z)}dV(z)\\
				&\quad\lesssim \int_{\mathbb{B}_n}|u(z)-v(z)|^{p(z)}dV(z)\\
				&\quad\quad\quad+\sum_{i=1}^n\int_{\mathbb{B}_n}|u(z)\varphi_i(z)-v(z)\psi_i(z)|^{p(z)}dV(z)\\
				&\quad=\rho_{p(\cdot)}(u-v)+\sum_{i=1}^n\rho_{p(\cdot)}(u\varphi_i-v\psi_i)<\infty,
			\end{split}
		\end{equation}
		since $u-v=(C_{u,\varphi}-C_{v,\psi})1\in A^{p(\cdot)}(\mathbb{B}_n)$ and $u\varphi_i-v\psi_i=(C_{u,\varphi}-C_{v,\psi})z_i\in A^{p(\cdot)}(\mathbb{B}_n)$ for $i=1,\cdots,n$.
		
		Combining \eqref{equa4.11}, \eqref{equa4.12} and \eqref{equa4.13}, we obtain
		\begin{equation*}
			\sup_{a\in\mathbb{B}_n}\frac{\mu_{u,\varphi,d}(E(a,\frac{s}{2}))}{(1-|a|^2)^{n+1}}<\infty \quad {\rm and} \quad 	\sup_{a\in\mathbb{B}_n}\frac{\mu_{v,\psi,d}(E(a,\frac{s}{2}))}{(1-|a|^2)^{n+1}}<\infty.
		\end{equation*}
		It follows from Lemma \ref{lemma2.06} that $\mu_{u,\varphi,d}$ and $\mu_{v,\psi,d}$ are Carleson measures for $A^{p(\cdot)}(\mathbb{B}_n)$.
		
		Furthermore, when $\varphi(z)\in E(a,\frac{s}{2})$, we have
		\begin{equation*}
			\begin{split}
				&|u(z)-v(z)|\left|\frac{1-\langle \varphi(z),a\rangle}{1-\langle \psi(z),a\rangle}\right|^{N}\\
				&\quad \leq \left|u(z)-v(z)\left(\frac{1-\langle \varphi(z),a\rangle}{1-\langle \psi(z),a\rangle}\right)^N\right|+|u(z)|\left|1-\left(\frac{1-\langle \varphi(z),a\rangle}{1-\langle \psi(z),a\rangle}\right)^N\right|\\
				&\quad \lesssim \left|u(z)-v(z)\left(\frac{1-\langle \varphi(z),a\rangle}{1-\langle \psi(z),a\rangle}\right)^N\right|+|u(z)|\left|\frac{\langle \varphi(z)-\psi(z),a\rangle}{1-\langle \psi(z),a\rangle}\right|.
			\end{split}
		\end{equation*}
		And by \eqref{equa2.1}, \eqref{equa2.3} and \eqref{equa2.4}, we have
		\[\left|\frac{1-\langle \varphi(z),a\rangle}{1-\langle \psi(z),a\rangle}\right|\gtrsim \frac{(1-|\varphi(z)|^2)(1-|\psi(z)|^2)}{|1-\langle \psi(z),\varphi(z)\rangle|^2}=1-d(z)^2\]
		for $z\in \varphi^{-1}(E(a,\frac{s}{2}))$. So combining \eqref{equa4.5} and \eqref{equa4.8}, we obtain
		\begin{equation*}
			\sup_{a\in\mathbb{B}_n}\frac{\int_{\varphi^{-1}(E(a,\frac{s}{2}))}|u(z)-v(z)|^{p(z)}(1-d(z))^{Np(z)}\omega_{\varphi}(z)dV(z)}{(1-|a|^2)^{n+1}}<\infty.
		\end{equation*}
		Then by Lemma \ref{lemma2.06}, we know that $\lambda_{\varphi,\alpha}$ is a Carleson measure for $A^{p(\cdot)}(\mathbb{B}_n)$ when $\alpha\geq Np^+\geq (n+1)p^+$. Similarly, $\lambda_{\psi,\alpha}$ is also a Carleson measure for $A^{p(\cdot)}(\mathbb{B}_n)$.
		
		{\it Compactness}: The proof for the compactness part is just a modification. One only needs to write ``$\lim\limits_{|a|\to 1}$" instead of ``$\sup\limits_{a\in\mathbb{B}_n}$" and write ''$\to 0$" instead of ``$<\infty$", respectively. We omit the details.
	\end{proof}
	
	\begin{prop}\label{proposition4.6}
		Let $p(\cdot)\in\mathcal{P}^{log}(\mathbb{B}_n)$, $u,v\in A^{p(\cdot)}(\mathbb{B}_n)$ and $\varphi,\psi\in H(\mathbb{B}_n)$. If the measures $\mu_{u,\varphi,d}^{(1)}$, $\mu_{v,\psi,d}^{(1)}$, $\lambda_{\varphi,\alpha}^{(1)}$ and $\lambda_{\psi,\alpha}^{(1)}$ are Carleson measures for $A^{p(\cdot)}(\mathbb{B}_n)$, then $C_{u,\varphi}-C_{v,\psi}$ is bounded on $A^{p(\cdot)}(\mathbb{B}_n)$.
	\end{prop}
	
	\begin{proof}
		Assume $\mu_{u,\varphi,d}^{(1)}$, $\mu_{v,\psi,d}^{(1)}$, $\lambda_{\varphi,\alpha}^{(1)}$ and $\lambda_{\psi,\alpha}^{(1)}$ are Carleson measures for $A^{p(\cdot)}(\mathbb{B}_n)$. Fix $0<s_1<s_2<1$, set $E_{s_1}=\{z\in\mathbb{B}_n:d(z)<s_1\}$. Let $f\in A^{p(\cdot)}(\mathbb{B}_n)$ with $\|f\|_{p(\cdot)}=1$. We write
		\begin{equation*}
			\begin{split}
				&\rho_{p(\cdot)}((C_{u,\varphi}-C_{v,\psi})f)\\
				&\quad =\int_{\mathbb{B}_n-E_{s_1}}\left|(C_{u,\varphi}-C_{v,\psi})f(z)\right|^{p(z)}dV(z)+\int_{E_{s_1}}\left|(C_{u,\varphi}-C_{v,\psi})f(z)\right|^{p(z)}dV(z)\\
				&\quad :=I(f)+II(f).
			\end{split}
		\end{equation*}
		By Lemma \ref{lemma2.3}, we have
		\begin{equation}\label{equa4.14}
			\begin{split}
				I(f)
				&\leq \frac{1}{s_1^{p^+}}\int_{\mathbb{B}_n-E_{s_1}}|(C_{u,\varphi}-C_{v,\psi})f(z)|^{p(z)}d(z)^{p(z)}dV(z)\\
				&\lesssim \int_{\mathbb{B}_n}\left(|u(z)|^{p(z)}|f(\varphi(z))|^{p(z)}+|v(z)|^{p(z)}|f(\psi(z))|^{p(z)}\right)d(z)^{p(z)}dV(z)\\
				& \lesssim \int_{\mathbb{B}_n}\left(|u(z)|^{p(z)}+|v(z)|^{p(z)}\right)dV(z)\\
				&\quad  +\int_{\mathbb{B}_n}|f(\varphi(z))|^{p(\varphi(z))}|u(z)|^{p(z)}\left[d(z)^{p(z)}\omega_{\varphi}(z)+1\right]dV(z)\\
				&\quad  +\int_{\mathbb{B}_n}|f(\psi(z))|^{p(\psi(z))}|v(z)|^{p(z)}\left[d(z)^{p(z)}\omega_{\psi}(z)+1\right]dV(z)\\
				&=\rho_{p(\cdot)}(u)+\rho_{p(\cdot)}(v)+\rho_{p(\cdot),\mu_{u,\varphi,d}^{(1)}}(f)+\rho_{p(\cdot),\mu_{v,\psi,d}^{(1)}}(f).
			\end{split}
		\end{equation}
		Now we estimate $II(f)$. Clearly,
		\begin{equation*}
			\begin{split}
				II(f)&\lesssim \int_{E_{s_1}}|u(z)-v(z)|^{p(z)}\left(|f(\varphi(z))|+|f(\psi(z))|\right)^{p(z)}dV(z)\\
				&\quad \quad+\int_{E_{s_1}}\left(|u(z)|+|v(z)|\right)^{p(z)}|f(\varphi(z))-f(\psi(z))|^{p(z)}dV(z)\\
				&:=II_1(f)+II_2(f).
			\end{split}
		\end{equation*}
		By Lemma \ref{lemma2.3} again, we have
		\begin{equation}\label{equa4.15}
			\begin{split}
				II_1(f)
				&\lesssim\int_{E_{s_1}}|u(z)-v(z)|^{p(z)}\left(|f(\varphi(z))|+|f(\psi(z))|\right)^{p(z)}(1-d(z))^{\alpha}dV(z)\\
				&\lesssim \int_{\mathbb{B}_n}|u(z)-v(z)|^{p(z)}dV(z)\\
				&\quad +\int_{\mathbb{B}_n}|f(\varphi(z))|^{p(\varphi(z))}|u(z)-v(z)|^{p(z)}(1-d(z))^{\alpha}(\omega_{\varphi}(z)+1)dV(z)\\
				&\quad +\int_{\mathbb{B}_n}|f(\psi(z))|^{p(\psi(z))}|u(z)-v(z)|^{p(z)}(1-d(z))^{\alpha}(\omega_{\psi}(z)+1)dV(z)\\
				&=\rho_{p(\cdot)}(u-v)+\rho_{p(\cdot),\lambda_{\varphi,\alpha}^{(1)}}(f)+\rho_{p(\cdot),\lambda_{\psi,\alpha}^{(1)}}(f).
			\end{split}
		\end{equation}
		When $z\in E_{s_1}$, by Lemma \ref{lemma2.2} and Lemma \ref{lemma2.5}, we obtain
		\begin{equation*}
			\begin{split}
				&|f(\varphi(z))-f(\psi(z))|^{p(\varphi(z))}\\
				&\quad \lesssim \left(\frac{d(z)}{(1-|\varphi(z)|^2)^{n+1}}\int_{E(\varphi(z),s)}|f(w)|dV(w)\right)^{p(\varphi(z))}\\
				&\quad \lesssim d(z)^{p(\varphi(z))}\left(1+\frac{1}{(1-|\varphi(z)|^2)^{n+1}}\int_{E(\varphi(z),s)}|f(w)|^{p(w)}dV(w)\right)\\
				&\quad \lesssim d(z)^{p(\varphi(z))}\frac{1}{(1-|\varphi(z)|^2)^{n+1}}
			\end{split}
		\end{equation*}
		Hence when $z\in E_{s_1}$ and $|f(\varphi(z))-f(\psi(z))|\geq 1$, we obtain
		\begin{equation*}
			\begin{split}
				&|f(\varphi(z))-f(\psi(z))|^{p(z)}=|f(\varphi(z))-f(\psi(z))|^{p(z)-p(\varphi(z))+p(\varphi(z))}\\
				&\quad \lesssim \left(1+d(z)^{p(z)}\omega_{\varphi}(z)\right)\left(1+\frac{1}{(1-|\varphi(z)|^2)^{n+1}}\int_{E(\varphi(z),s)}|f(w)|^{p(w)}dV(w)\right).
			\end{split}
		\end{equation*}
		Similarly,
		\begin{equation*}
			\begin{split}
				&|f(\varphi(z))-f(\psi(z))|^{p(z)}\\
				&\quad \lesssim\left(1+d(z)^{p(z)}\omega_{\psi}(z)\right)\left(1+\frac{1}{(1-|\psi(z)|^2)^{n+1}}\int_{E(\psi(z),s)}|f(w)|^{p(w)}dV(w)\right)
			\end{split}
		\end{equation*}
		when $z\in E_{s_1}$ and $|f(\varphi(z))-f(\psi(z))|>1$. Therefore, by \eqref{equa2.3} and Fubini's Theorem, we obtain
		\begin{equation}\label{equa4.16}
			\begin{split}
				II_2(f)	&\lesssim \int_{\mathbb{B}_n}\frac{|u(z)|^{p(z)}\left[d(z)^{p(z)}\omega_{\varphi}(z)+1\right]}{(1-|\varphi(z)|^2)^{n+1}}\int_{E(\varphi(z),s)}|f(w)|^{p(w)}dV(w)dV(z)\\
				&\quad +\int_{\mathbb{B}_n}\frac{|v(z)|^{p(z)}\left[d(z)^{p(z)}\omega_{\psi}(z)+1\right]}{(1-|\psi(z)|^2)^{n+1}}\int_{E(\psi(z),s)}|f(w)|^{p(w)}dV(w)dV(z)\\
				&\quad +\int_{\mathbb{B}_n}\left(|u(z)|^{p(z)}+|v(z)|^{p(z)}\right)dV(z)+1\\
				&\lesssim 1+\rho_{p(\cdot)}(u)+\rho_{p(\cdot)}(v)+\sup_{w\in\mathbb{B}_n}\frac{\mu_{u,\varphi,d}(E(w,s))+\mu_{v,\psi,d}(E(w,s))}{(1-|w|^2)^{n+1}}.
			\end{split}
		\end{equation}
		Combining \eqref{equa4.14}, \eqref{equa4.15} and \eqref{equa4.16}, we obtain
		\begin{equation*}
			\begin{split}
				&\rho_{p(\cdot)}((C_{u,\varphi}-C_{v,\psi})f)\\
				&\quad \lesssim 1+\rho_{p(\cdot)}(u)+\rho_{p(\cdot)}(v)+\rho_{p(\cdot)}(u-v)\\
				&\quad\quad  +\rho_{p(\cdot),\mu_{u,\varphi,d}^{(1)}}(f)+\rho_{p(\cdot),\mu_{v,\psi,d}^{(1)}}(f)+\rho_{p(\cdot),\lambda_{\varphi,\alpha}^{(1)}}(f)+\rho_{p(\cdot),\lambda_{\psi,\alpha}^{(1)}}(f)\\
				&\quad\quad +\sup_{w\in\mathbb{B}_n}\frac{\mu_{u,\varphi,d}^{(1)}(E(w,s))+\mu_{v,\psi,d}^{(1)}(E(w,s))}{(1-|w|^2)^{n+1}}
			\end{split}
		\end{equation*}
		for any $f\in A^{p(\cdot)}(\mathbb{B}_n)$ with $\|f\|_{p(\cdot)}=1$. It follows that $C_{u,\varphi}-C_{v,\psi}$ is bounded on $A^{p(\cdot)}(\mathbb{B}_n)$.
	\end{proof}
	
	\begin{thm}\label{theorem4.7}
		Let $p(\cdot)\in\mathcal{P}^{log}(\mathbb{B}_n)$, $u,v\in H(\mathbb{B}_n)$ and $\varphi,\psi\in S(\mathbb{B}_n)$. Suppose $p(z)\geq \max\{p(\varphi(z)),p(\psi(z))\} $ a.e. Then $C_{u,\varphi}-C_{v,\psi}$ is bounded on $A^{p(\cdot)}(\mathbb{B}_n)$ if and only if the measures $\mu_{u,\varphi,d}$, $\mu_{v,\psi,d}$, $\lambda_{\varphi,\alpha}$ and $\lambda_{\psi,\alpha}$ are Carleson measures for $A^{p(\cdot)}(\mathbb{B}_n)$.
	\end{thm}
	
	\begin{proof}
		The necessity has been proved in Theorem \ref{theorem4.5}. And following the arguments in the proof of Proposition \ref{proposition4.6} step by step, we can prove the sufficiency. 
		
		In fact, assume $\mu_{u,\varphi,d}$, $\mu_{v,\psi,d}$, $\lambda_{\varphi,\alpha}$ and $\lambda_{\psi,\alpha}$ are Carleson measures for $A^{p(\cdot)}(\mathbb{B}_n)$ and $p(z)\geq \max\{p(\varphi(z)),p(\psi(z))\} $. Modifying the proof of Proposition \ref{proposition4.6}, for any $f\in A^{p(\cdot)}(\mathbb{B}_n)$, by Lemma \ref{lemma2.3}, we obtain
		\begin{equation}\label{equa4.17}
			I(f)\lesssim \rho_{p(\cdot),\mu_{u,\varphi,d}}(f)+\rho_{p(\cdot),\mu_{v,\psi,d}}(f)\lesssim 1,
		\end{equation}
		and
		\begin{equation}\label{equa4.18}
			II_1(f)\lesssim \rho_{p(\cdot),\lambda_{\varphi,\alpha}}(f)+\rho_{p(\cdot),\lambda_{\psi,\alpha}}(f)\lesssim 1.
		\end{equation}
		By Lemma \ref{lemma2.5} and Fubini's Theorem, we obtain
		\begin{equation}\label{equa4.19}
			II_2(f)\lesssim 1+\sup_{w\in\mathbb{B}_n}\frac{\mu_{u,\varphi,d}(E(w,s))+\mu_{v,\psi,d}(E(w,s))}{(1-|w|^2)^{n+1}}\lesssim 1.
		\end{equation}
		Therefore, combining \eqref{equa4.17}, \eqref{equa4.18} and \eqref{equa4.19}, and by condition (ii) in Lemma \ref{lemma1.1}, we get
		\begin{equation*}
			\begin{split}
				\|(C_{u,\varphi}-C_{v,\psi})f\|_{p(\cdot)}&\leq \rho_{p(\cdot)}((C_{u,\varphi}-C_{v,\psi})f)+1\\
				&\lesssim I(f)+II_1(f)+II_2(f)+1\lesssim 1
			\end{split}
		\end{equation*}
		for any $f\in A^{p(\cdot)}(\mathbb{B}_n)$ with $\|f\|_{p(\cdot)}=1$. Then a scaling argument yields the boundedness of $C_{u,\varphi}-C_{v,\psi}$ on $A^{p(\cdot)}(\mathbb{B}_n)$.
	\end{proof}
	
	\begin{rem}
		In the case of variable exponent, it remains an open question whether the sufficiency stated in Theorem 4.7 applies to the compactness of the operator $C_{u,\varphi}-C_{v,\psi}$, even though it is valid in the scenario of constant exponent.
	\end{rem}

	
	\subsection*{Acknowledgment}
	The authors thank the referees for the comments and suggestions have led to the improvement of the paper.
	
	\subsection*{Declaration}
	The authors declare that they have no conflicts of interest. No data was used for the research described in this article.



\end{document}